\documentclass[12pt]{article}
\usepackage{times}
\usepackage{booktabs}
\usepackage{appendix}
\usepackage{pifont}
\usepackage{floatrow}
\usepackage{caption}
\usepackage{mathrsfs}
\usepackage[fleqn]{amsmath}
\usepackage{amsfonts,amsthm,amssymb,mathrsfs,bbding}
\usepackage{txfonts}
\usepackage{graphics,multicol}
\usepackage{graphicx}
\usepackage{color}
\usepackage{caption}
\usepackage{subfigure}
\usepackage{amsmath,graphicx}
\usepackage{cite}
\usepackage{latexsym,bm}
\usepackage{longtable}
\usepackage{supertabular}
\usepackage{enumerate}
\usepackage{amsmath}
\evensidemargin=\oddsidemargin\topmargin=-1.5cm

\newtheorem{lem}{Lemma}[section]

\newtheorem{thm}{Theorem}[section]

\newtheorem{remark}{Remark}
\newtheorem*{claim}{Claim}

\theoremstyle{definition}

\addtocounter{section}{0}
\allowdisplaybreaks[4]

\begin{document}
\title{Some classes of connected signed graphs with girth $g$ and negative inertia index $\lceil\frac{g}{2}\rceil+1$ \footnote{This work is sponsored by the National Natural Science Foundation of China (No. 12561064).}}
\author{{BeiYan Liu,\ \   Fang Duan \footnote{Email: fangbing327@126.com}}\\[2mm]
\small School of Mathematics Science,  Xinjiang Normal University,\\
\small Urumqi, Xinjiang 830017, P.R. China}
\date{}
\maketitle {\flushleft\large\bf Abstract:}
Let $\Gamma$ be a signed graph. The number of negative eigenvalues of the adjacency matrix of $\Gamma$ is called the negative inertia index of $\Gamma$, which is denoted by
$i_-(\Gamma)$. The length of the shortest cycle contained in $\Gamma$ is called the girth of $\Gamma$, and it is denoted by $g$. In this paper, we give some classes of connected  signed graphs $\Gamma$ which satisfy the condition $i_-(\Gamma)=\lceil\frac{g}{2}\rceil+1$.

\begin{flushleft}
\end{flushleft}

\section{Introduction}
Let $G=(V(G), E(G))$ be a graph consisting of vertices $V(G)=\{v_1, v_2, ..., v_n\}$ and edges $E(G)$. For $u, v\in V(G)$ and $S\subseteq V(G)$, $u\sim v$ denotes $uv\in E(G)$, $N_S(v)=\{u\in S\mid uv\in E(G)\}$ denotes the \emph{neighborhood} of $v$ in $S$, and $|N_S(v)|$ is called the \emph{degree} of $v$ in $S$, which is written as $d_S(v)$. In particular, if $S=V(G)$, then $N_S(v), d_S(v)$ can be replaced by $N(v), d(v)$, respectively. Following standard graph-theoretic conventions, we denote by $K_{n_1, \ldots, n_l}$ the \emph{complete multipartite graph} with $l$ parts of sizes $n_1, \ldots, n_l$, $K_{1,n}$ the \emph{star}, and the vertex of degree $n$ is called the \emph{central vertex} of $K_{1,n}$. $C_n$ and $P_n$ denote the \emph{cycle} and \emph{path} on $n$ vertices, respectively.

The \emph{adjacency matrix} $A(G)=(a_{ij})_{n\times n}$ of $G$ is defined as follows: $a_{ij}=1$ if $v_i$ is adjacent to $v_j$, and $a_{ij}=0$ otherwise. Owing to the real symmetry of $A(G)$, all its eigenvalues are real numbers, which can be ordered in non-increasing sequence as ($\lambda_1(G) \geq \lambda_2(G) \geq \cdots \geq \lambda_n(G)$ ); this sequence is commonly termed the spectrum of $G$. Three critical inertia parameters are derived from this spectrum: the positive inertia index $i_+(G)$ (count of positive eigenvalues), the negative inertia index $i_-(G)$ (count of negative eigenvalues), and the nullity $\eta(G)$ (multiplicity of the zero eigenvalue). These parameters occupy a pivotal position at the intersection of matrix theory and graph theory, with direct applications in analyzing the stability of chemical molecules (e.g.,predicting $\pi$-electronic structures of conjugated compounds) and providing algebraic criteria for graph classification. Consequently, they have remained a focal point of spectral graph theory research over the past three decades \cite{Fang.D1,Fang.D2, M.R.Oboudi3,GuiHai.Yu}.

A \emph{signed graph} $\Gamma=(G, \sigma)$ consists of $G$ (called its \emph{underlying graph}) and a sign function $\sigma: E(G)\rightarrow \{+, -\}$. If $\sigma(e)=+$ (resp., $\sigma(e)=-$) for each $e\in E(G)$, then we denote the signed graph by $\Gamma=(G, +)$ (resp., $\Gamma=(G, -))$. Given a subset $\{u_1, u_2, \ldots ,u_t\}=S\subseteq V(\Gamma)$, the subgraph of $\Gamma$ induced by $S$, written as $\Gamma[S]$, is defined to be the signed graph with vertex set $S$ and edge set $\{u_iu_j\in E(\Gamma)\mid u_i\in S$ and $u_j\in S\}$, where the sign of $u_iu_j$ in $\Gamma[S]$ is the same as it in $\Gamma$. The \emph{adjacency matrix} of $\Gamma=(G, \sigma)$ is defined as $A(\Gamma)=(a_{u,v}^\sigma)_{u,v\in V(G)}$, where $a_{u,v}^\sigma=\sigma(uv) \cdot 1$ if $uv\in E(G)$ and $a_{u,v}^\sigma=0$ otherwise. The eigenvalues of $A(\Gamma)$ are called the \emph{eigenvalues} of $\Gamma$. The inertia indices, nullity and other basic notations of signed graphs are similarly to that of simple graphs.

The \emph{sign} of a cycle $C^\sigma$ is defined by $sgn(C^\sigma)=\prod_{e\in E(C)}\sigma(e)$. If $sgn(C^\sigma)=+$ (resp., $sgn(C^\sigma)=-$), then $C^\sigma$ is said to be \emph{positive} (resp., \emph{negative}). A signed graph $\Gamma$ is said to be \emph{balanced} if all of its cycles are positive, and \emph{unbalanced} otherwise. In particular, an acyclic signed graph is balanced. A function $\theta: V(\Gamma)\rightarrow \{+, -\}$ is called a \emph{switching function} of $\Gamma$. Switching $\Gamma$ by $\theta$, we obtain a new signed graph $\Gamma^\theta=(G, \sigma^\theta)$ whose sign function $\sigma^\theta$ is defined by $\sigma^\theta(xy)=\theta(x)\sigma(xy)\theta(y)$ for any $xy\in E(G)$. Two signed graphs $\Gamma_1$ and $\Gamma_2$ are said to be \emph{switching equivalent}, denoted by $\Gamma_1 \sim \Gamma_2$, if there exists a switching function $\theta$ such that $\Gamma_2=\Gamma_1^{\theta}$.

Signed graphs were first formalized by Harary\cite{F.Harary}, and their spectral properties-combining the intuitive structure of graphs with the algebraic power of matrix analysis-have since emerged as a core research direction in spectral graph theory \cite{Fang.D,W.H.Haemers,M.Brunetti,Q.Wu}. Belardo et al. \cite{F.Belardo}\ provided a comprehensive overview of fundamental signed graph theory, including inertia indices calculation methods and girth-related properties, while proposing several open conjectures on signed graph structural characteristics that have guided subsequent research. Previously, Duan et.al \cite{Fang.D1,Fang.D2} investigated the inertia indices of simple connected graphs with fixed girth; in the field of signed graphs, Liu et.al\cite{L.D} systematically characterized the extremal connected signed graphs satisfying $i_-(\Gamma)=\lceil\frac{g}{2}\rceil-1$ and $i_-(\Gamma)=\lceil\frac{g}{2}\rceil$. In the parallel research thread of nullity, Suliman Khan \cite{K.S} has provided characterizations for signed graphs $\Gamma$ that meet the nullity conditions $\eta(\Gamma)=n(\Gamma)-g(\Gamma)$ and $ \eta(\Gamma)=n(\Gamma)-g(\Gamma)-2$. It has laid a significant foundation for the research on the correlation between the nullity and girth of signed graphs. Focusing on the relationship between the negative inertia index and girth of signed graphs, this paper further expands on the research findings of \cite{L.D}. It conducts in-depth studies on the signed graphs that satisfy the condition $i_-(\Gamma)= \lceil\frac{g}{2}\rceil+1$. This research not only supplements the research dimension of inertia index but also complements Suliman Khan's research on nullity, jointly enriching the theoretical system regarding the relationship between the core indices and girth of signed graphs.

The rest of the paper is organized as follows: Section 2 reviews essential basic concepts and existing lemmas; Section 3 presents the main theorem for the characterization of signed graphs with $i_-(\Gamma)=\lceil\frac{g}{2}\rceil+1$ along with its detailed proof.

\section{Preliminaries}
\begin{lem}[\cite{R.A.Horn}]\label{lem-2-0-0}
Let $A$ be a real matrix of order $n$ and $\lambda_1, \lambda_2, \ldots, \lambda_n$ be all eigenvalues of $A$. Then $det(A)=\lambda_1 \lambda_2\cdots\lambda_n$.
\end{lem}

\begin{lem}\label{lem-2-0}
(Sylvester's law of inertia) If two real symmetric matrices $A$ and $B$ are congruent, then they have the same positive (resp., negative)
inertia index, the same nullity.
\end{lem}

\begin{lem}[\cite{L.D}]\label{lem-2-5}
Let $\Gamma=(G, \sigma)$ be a signed graph with girth $g$ and $C_g^\sigma$ be a shortest signed cycle of $\Gamma$. Suppose $y,y'\in V(C_g^\sigma)$ and there exists an induced path $P^\sigma$ of length $t$ from $y$ to $y'$ vertex-disjoint to $C_g^\sigma$. Then $\lceil\frac{g}{2}\rceil\leq t$. Furthermore, if $g\geq 5$, then each vertex in $V(\Gamma)\setminus V(C_g^\sigma)$ is adjacent to at most one vertex of $V(C_g^\sigma)$.
\end{lem}

\begin{lem}[\cite{G.H.Yu}]\label{lem-2-2}
Let $\Gamma$ be a signed graph containing a pendant vertex $u$, and let $\Gamma'$ be the induced signed subgraph of $\Gamma$ obtained by deleting $u$ together with the vertices adjacent to it. Then $i_+(\Gamma)=i_+(\Gamma')+1$ and $i_-(\Gamma)=i_-(\Gamma')+1$.
\end{lem}
\begin{lem}\label{lem-2-1-1}
Let $\Gamma'$ be an induced signed subgraph of signed graph $\Gamma$. Then $i_+(\Gamma)\geq i_+(\Gamma')$ and $i_-(\Gamma)\geq i_-(\Gamma')$.
\end{lem}

\begin{lem}[\cite{GuiHai.Yu}, \cite{G.H.Yu}, \cite{G.H.Yu1}]\label{lem-2-1}
Let $C_n^\sigma$, $P_n^\sigma$ be a signed cycle, a signed path of order $n$, respectively.
\begin{enumerate}[(1)]
\item If $C_n^\sigma$ is balanced, then $i_-(C_n^\sigma)=\left\{\begin{array}{ll}\lceil\frac{n}{2}\rceil-1 &n\equiv 0,1(\bmod 4) \\ \lceil\frac{n}{2}\rceil & n\equiv 2,3(\bmod 4)\end{array}\right.$;
\item If $C_n^\sigma$ is unbalanced, then $i_-(C_n^\sigma)=\left\{\begin{array}{ll}\lceil\frac{n}{2}\rceil-1 &n\equiv 2,3(\bmod 4) \\ \lceil\frac{n}{2}\rceil & n\equiv 0,1(\bmod 4)\end{array}\right.$;
\item $i_-(P_n^\sigma)=\lfloor\frac{n}{2}\rfloor$.
\end{enumerate}
\end{lem}

For $x\not\in Y$, a family of $k$ internally disjoint $(x, Y)$-paths whose terminal vertices are distinct is referred to as a $k$-\emph{fan} from $x$ to $Y$. Furthermore, if all the paths of a $k$-fan have length $l$, then we call it a $k$-fan of length $l$.

\begin{figure}[htbp]
    \centering
    \includegraphics[width=1\linewidth]{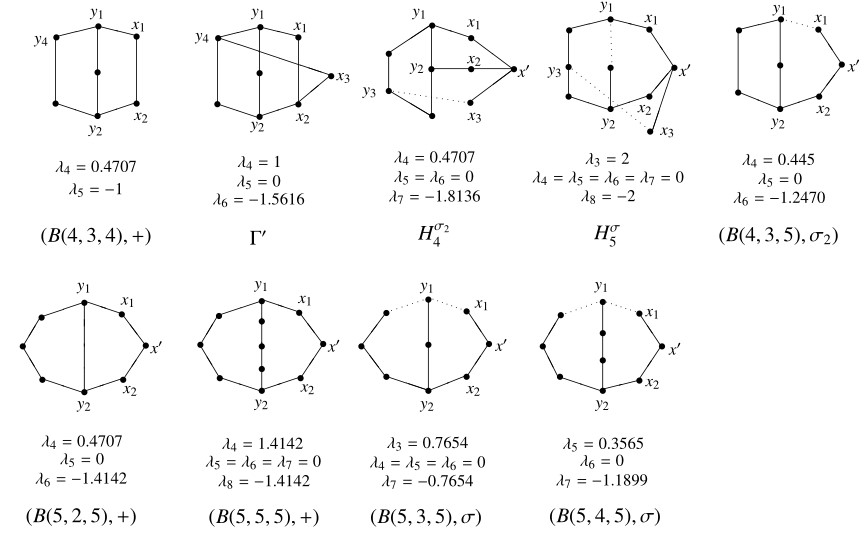}
    \caption{The signed graphs $(B(4,3,4),+)$, $\Gamma'$, $H_4^{\sigma_2}$, $H_5^{\sigma}$, $(B(4,3,5),\sigma_2)$, $(B(5,2,5),+)$, $(B(5,3,5),\sigma)$, $(B(5,4,5),\sigma)$ and $(B(5,5,5),+)$}
    \label{fig-0}
    \end{figure}

\begin{lem}\label{lem-2-1-5}
Let $\Gamma=(G, \sigma)$ be a signed graph with girth $g$ and $C_g^\sigma$ be a shortest signed cycle of $\Gamma$. If there exists a $k$-fan of length $l$ from $x$ to $V(C_g^\sigma)$,  then $\lfloor\frac{g}{k}\rfloor+2l\geq g$.
\end{lem}
\begin{proof}
Because this $k$-fan has $k$ different terminal vertices in $V(C_g^\sigma)$, there must be two of these terminal vertices whose distance on $V(C_g^\sigma)$ is less than or equal to $\lfloor\frac{g}{k}\rfloor$. Hence we can find a signed cycle of length at most $\lfloor\frac{g}{k}\rfloor+2l$ in $\Gamma$, and hence, the desired inequality holds.
\end{proof}

\begin{thm}[\cite{Yaoping.Hou}]\label{thm-1-0}
Let $\Gamma=(G, \sigma)$ be a signed graph. Then $\Gamma$ is balanced if and only if $\Gamma \backsim (G,+)$.
\end{thm}

\begin{lem}[\cite{FanY.Z}] \label{lem-2-11}
Let $\Gamma$ be an unbalanced signed unicyclic graph of order $n$. Then $\Gamma$ is switching equivalent to a signed unicyclic graph of order $n$ with exactly one (arbitrary) negative edge on the cycle.
\end{lem}

Let $C^\sigma$ be a signed cycle of $\Gamma$ and $v\in V(\Gamma)\setminus V(C^\sigma)$. Denoted by $d(v, x)$ the length of a shortest path between $v$ and $x$ and let $d(v, C^\sigma)=min\{d(v,x)\mid x\in V(C^\sigma)\}$. Defined
$$N_r(v,C^\sigma) =\left\{v\in V\left(\Gamma\right)\setminus V( C^\sigma)\mid d(v,C^\sigma)=r\right\},$$
where $r$ is a positive integer, and $|N_r(v, C^\sigma)|$ denote the number of vertices in $N_r(v, C^\sigma)$. If a signed unicyclic graph $\Gamma$ satisfies (a) $N_r(v, C^\sigma)=\emptyset$ for $r\geq 2$; (b) $N_1(v, C^\sigma)$ is an independent set, and each vertex of $N_1(v, C^\sigma)$ is adjacent to exactly one vertex of $V(C^\sigma)$, then $\Gamma$ is called a \emph{canonical signed unicyclic graph}. Set $H^\sigma$ is an induced signed subgraph of the canonical signed unicyclic graph $\Gamma$. We call $H^\sigma$ a \emph{attached pendant star} of $\Gamma$ if $H^\sigma$ is a star and its center is the only vertex which have exactly two neighbors in $C^\sigma$, this center vertex of $H^\sigma$ is called the \emph{major vertex}.

\begin{thm}[\cite{L.D}]\label{thm-1-1}
Let $\Gamma$ be a canonical signed unicyclic graph with girth $g$ and the unique signed cycle $C_g^\sigma$. Then, the following statements hold:
\begin{enumerate}[(1)]
\setlength{\itemsep}{0pt}
\item If $\Gamma$ is a signed cycle, then $i_-(\Gamma)=\lceil \frac{g}{2}\rceil$ if and only if $\Gamma\cong C_g^\sigma$, where $C_g^\sigma$ is balanced and $g\equiv 2,3(\bmod\ 4)$, or unbalanced and $g\equiv 0,1(\bmod\ 4)$;
\item If $\Gamma$ is not a signed cycle and $g\equiv 1,3(\bmod\ 4)$, then $i_-(\Gamma)=\lceil \frac{g}{2}\rceil$ if and only if $\Gamma$ has one or more attached pendant stars such that exactly one path between any two major vertices of $V(C_g^\sigma)$ has even order;
\item If $\Gamma$ is not a signed cycle and $g\equiv 0,2(\bmod\ 4)$, then $i_-(\Gamma)=\lceil \frac{g}{2}\rceil$ if and only if $\Gamma$ has one or more attached pendant stars such that all paths between any two major vertices of $V(C_g^\sigma)$ have odd order.
\end{enumerate}
\end{thm}

\begin{thm}[\cite{L.D}]\label{thm-1-2}
Let $\Gamma=(G,\sigma)$ be a connected signed graph with girth $g\geq 4$, and there exists no signed cycle $C_g^\sigma$ of $\Gamma$ satisfying $i_-(C_g^\sigma)=\lceil\frac{g}{2}\rceil$. Suppose that $\Gamma$ is not a canonical signed unicyclic graph. Then $i_-(\Gamma)=\lceil\frac{g}{2}\rceil$ if and only if $\Gamma$ is isomorphic to one of the following signed graphs:
\setlength{\itemsep}{0pt}
\item (1) The signed graphs with girth 4 obtained from $P_4^\sigma$, $P_5^\sigma$, balanced $C_5^\sigma$ or unbalanced $C_6^\sigma$ by adding twin vertices;
\item (2) The signed graphs obtained by joining a vertex of signed cycle $C_g^\sigma$ to the center of a star $K_{1,r}$, where $g\equiv 0,1(\bmod\ 4)$ if $C_g^\sigma$ is balanced and $g\equiv 2,3(\bmod\ 4)$ if $C_g^\sigma$ is unbalanced;
\item (3) $(B(4,3,4),+)$, $\Gamma'$, $H_4^{\sigma_2}$, $H_5^{\sigma}$, $(B(4,3,5),\sigma_2)$, $(B(5,2,5),+)$, $(B(5,3,5),\sigma)$, $(B(5,4,5),\sigma)$ and $(B(5,5,5),+)$ (see Fig. \ref{fig-0}).
\end{thm}

\begin{remark}\label{re-1}
The girths of the graphs shown in Theorem \ref{thm-1-2} are all less than 6.
\end{remark}
\section{Connected signed graphs with given negative inertia index and given girth}
In this section, we establish three main theorems.

\begin{thm}\label{thm-2-1}
Let $\Gamma$ be a canonical signed unicyclic graph with girth $g$ and the unique signed cycle of $\Gamma$ is $C_g^\sigma$. Then, the following statements hold:
\begin{enumerate}[(1)]
\setlength{\itemsep}{0pt}
\item If $g\equiv 1,3(\bmod\ 4)$, then $i_-(\Gamma)=\lceil \frac{g}{2}\rceil+1$ if and only if $\Gamma$ has at least three attached pendant stars, and there are exactly three paths of even order between any two major vertices in $V(C_g^\sigma)$;

\item If $g\equiv 0,2(\bmod\ 4)$, then $i_-(\Gamma)=\lceil \frac{g}{2}\rceil+1$ if and only if $\Gamma$ has at least two attached pendant stars, and there are exactly two paths of even order between any two major vertices in $V(C_g^\sigma)$.
\end{enumerate}
\end{thm}

\begin{proof}
Clearly, $\Gamma\not\cong C_g^\sigma$ by Lemma \ref{lem-2-1}. Hence there are $k\geq 1$ attached pendant stars in $\Gamma$. Let $P_{l_1}^\sigma, \ldots, P_{l_k}^\sigma$ be the paths obtained upon removing the $k$ attached pendant stars from $\Gamma$. Then, $g=k+l_1 +\cdots +l_k$. On the other hand, because there are $k$ attached pendant stars in $\Gamma$, we have $i_-(\Gamma)=k+i_-(P_{l_1}^\sigma)+\cdots +i_-(P_{l_k}^\sigma)$ by Lemma \ref{lem-2-2}. The following two cases can be considered:

{\flushleft\bf Case 1.} $g\equiv 1,3$ (mod 4).

Since $g$ is an odd in this case, we have $i_-(\Gamma)=\lceil \frac{g}{2}\rceil+1$ if and only if
$$k+i_-(P_{l_1}^\sigma)+\cdots +i_-(P_{l_k}^\sigma)=i_-(\Gamma)=\lceil \frac{g}{2}\rceil+1=\frac{g+1}{2}+1=\frac{k+l_1 +\cdots +l_k+3}{2}.$$
Then,
$$k-3=[l_1-2i_-(P_{l_1}^\sigma)]+[l_2-2i_-(P_{l_2}^\sigma)]+\cdots +[l_k-2i_-(P_{l_k}^\sigma)].$$
According to Lemma \ref{lem-2-1}, $l_j-2i_-(P_{l_j}^\sigma)=0$ if $l_j$ is even and $l_j-2i_-(P_{l_j}^\sigma)=1$ if $l_j$ is odd for $j\in \{1,\ldots,k\}$. Consequently, $i_-(\Gamma)=\lceil \frac{g}{2}\rceil+1$ if and only if
three of $l_1, \ldots, l_k$ are even, as desired.

{\flushleft\bf Case 2.} $g\equiv 0, 2$ (mod 4).

Since $g$ is an even, we have $i_-(\Gamma)=\lceil \frac{g}{2}\rceil+1$ if and only if
$$k+i_-(P_{l_1}^\sigma)+\cdots +i_-(P_{l_k}^\sigma)=i_-(\Gamma)=\lceil \frac{g}{2}\rceil+1=\frac{g}{2}+1=\frac{k+l_1 +\cdots +l_k+2}{2}.$$
Then,
$$k-2=[l_1-2i_-(P_{l_1}^\sigma)]+[l_2-2i_-(P_{l_2}^\sigma)]+\cdots +[l_k-2i_-(P_{l_k}^\sigma)].$$
By a similar discussion as Case 1, $i_-(\Gamma)=\lceil \frac{g}{2}\rceil+1$ holds if and only if two of $l_1,\ldots, l_k$ are even. It follows the desired conclusion.
\end{proof}

Let $G_1$ be a graph containing vertex $u$, and let $G_2$ be a graph of order $n$ that is disjoint from $G_1$. For $1\leq k\leq n$, the \emph{$k$-joining graph} of $G_1$ and $G_2$ with respect to $u$ is obtained from $G_1\cup G_2$ by joining $u$ and any $k$ vertices of $G_2$, and it is denoted by $G_1(u)\odot^kG_2$. Clearly, the graph $G_1(u)\odot^kG_2$ is not uniquely determined when $n>k$.

\begin{figure}[htbp]
    \centering
    \includegraphics[width=0.9\linewidth]{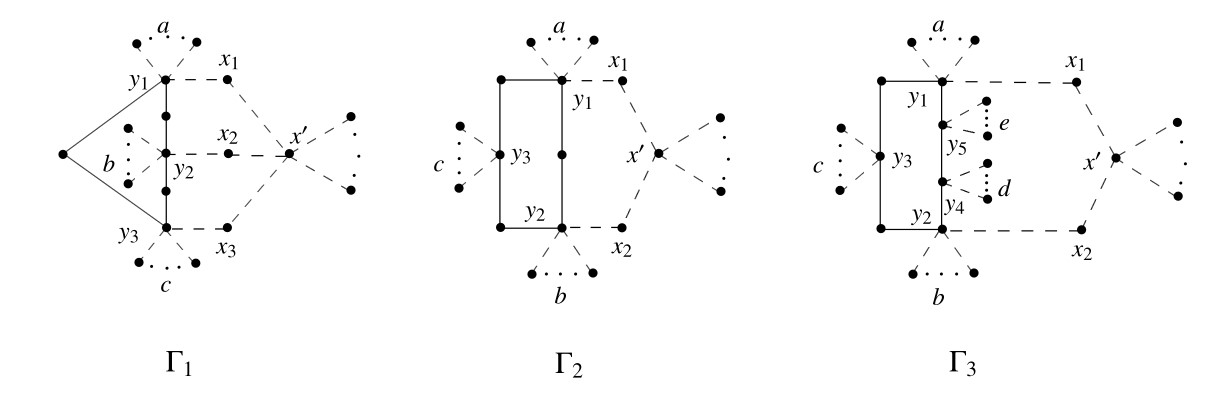}
    \caption{The signed graphs $\Gamma_1$, $\Gamma_2$ and $\Gamma_3$}
    \label{fig-1}
    \end{figure}

We present a method for representing signed graphs that will be used in this section. Two vertices connected by a solid line (resp., dotted line) indicate that the corresponding edge is positive (resp., negative), while those connected by a dashed line indicate that the sign of the edge is undetermined.

\begin{thm}\label{thm-2-2}
Let $\Gamma=(G,\sigma)$ be a connected triangle-free signed graph with girth $g\equiv 2,3(\bmod~4)$. Suppose that there exists a balanced signed cycle $C_g^\sigma$ in $\Gamma$ and $\Gamma$ is not a canonical signed unicyclic graph. If $N_3(v,C_g^\sigma)\neq\emptyset$, then $i_-(\Gamma)=\lceil\frac{g}{2}\rceil+1$ if and only if $\Gamma$ is isomorphic to one of the following signed graphs:
\setlength{\itemsep}{0pt}
{\item (1) The signed graphs $\Gamma_1$ and $\Gamma_2$, where $a,b,c$ are non-negative integers;
\item (2) The signed graph $\Gamma_3$, where $a,b,c,d,e$ are non-negative integers and $d, e$ cannot both be non-zero;
\item (3) The signed graph $K_{1, t}^\sigma(x')\odot^1\Gamma'$, where $x'$ is the center vertex of $K_{1, t}^\sigma$, $\Gamma'$ is one of the signed graphs shown in Theorem \ref{thm-1-1}(2)(3) with balanced signed cycle $C_g^\sigma$ for $g\equiv 2,3(\bmod~4)$, and $x'$ is adjacent to exactly one pendant vertex of $\Gamma'$. }
\end{thm}

\begin{proof}
By a simple operation, we can verify that the negative inertia index of each of the signed graphs shown in (1) (2) and (3) is $\lceil\frac{g}{2}\rceil+1$. Conversely,
let $\Gamma$ be a connected triangle-free signed graph with girth $g\equiv 2,3(\bmod~4)$ and $i_-(\Gamma)=\lceil\frac{g}{2}\rceil+1$. Then $g\geq 6$, and it may be assumed that each edge of $C_g^\sigma$ is positive by Theorem \ref{thm-1-0}. Note $i_-(C_g^\sigma)=\lceil\frac{g}{2}\rceil$ by Lemma \ref{lem-2-1}. We have $N_4(v,C_g^\sigma)=\emptyset$. In fact, on the contrary, suppose that $x'''\in N_4(v,C_g^\sigma)\neq\emptyset$. Then, we may assume $x'''\sim x''\in N_3(v,C_g^\sigma)$, $x''\sim x'\in N_2(v,C_g^\sigma)$ and $x'\sim x\in N_1(v,C_g^\sigma)$.  According to Lemmas \ref{lem-2-2} and \ref{lem-2-1-1},  a contradiction
$$i_-(\Gamma)\geq i_-(\Gamma[V(C_g^\sigma)\cup\{x,x',x'',x'''\}])=i_-(\Gamma[V(C_g^\sigma)])+2=\lceil\frac{g}{2}\rceil +2>\lceil\frac{g}{2}\rceil+1$$ arise.
Hence, $N_i(v,C_g^\sigma)=\emptyset$ for $i\geq 4$. Note $N_3(v,C_g^\sigma)\neq\emptyset$. We give the following claim.

\begin{claim}
$|N_2(v,C_g^\sigma)|=1$.
\end{claim}
\begin{proof}
Let $x''\in N_3(v,C_g^\sigma)\neq\emptyset$ and $x''\sim x_1'\in N_2(v,C_g^\sigma)$. Suppose $|N_2(v,C_g^\sigma)|\geq 2$. Set $x_1'\neq x_2'\in N_2(v,C_g^\sigma)$ and $x_2'\sim x_2\in N_1(v,C_g^\sigma)$. If $x_2'\not\sim x''$, then $$i_-(\Gamma)\geq i_-(\Gamma[V(C_g^\sigma)\cup \{x_2,x_1',x_2',x''\}])=i_-(C_g^\sigma)+2>\left \lceil \frac{g}{2}\right\rceil+1$$ by Lemmas \ref{lem-2-2} and \ref{lem-2-1-1}, a contradiction. Hence $x_2'\sim x''$, and thus $x_1'\not\sim x_2'$. Furthermore, there exists $x_2\neq x_1\in N_1(v,C_g^\sigma)$ such that $x_1'\sim x_1$ because $g\geq 6$. A contradiction $i_-(\Gamma)\geq i_-(\Gamma[V(C_g^\sigma)\cup \{x_1, x_2, x_1',x_2'\}])=i_-(C_g^\sigma)+2>\left \lceil \frac{g}{2}\right\rceil+1$ also follows from Lemmas \ref{lem-2-2} and \ref{lem-2-1-1}. Then $|N_2(v,C_g^\sigma)|=1$.
\end{proof}

We may assume $N_2(v,C_g^\sigma)=\{x'\}$. Since each vertex of $N_3(v,C_g^\sigma)$ is adjacent to $x'$ and $\Gamma$ is triangle-free, we get that $N_3(v,C_g^\sigma)$ is an independent set. Assume that $x'$ is adjacent to four vertices $x_1,x_2,$ $x_3,x_4$ of $N_1(v,C_g^\sigma)$. Then, there is a $4$-fan of length 2 from $x'$ to $V(C_g^\sigma)$ since $g\geq 6$. Consequently, $\lfloor\frac{g}{4}\rfloor+4\geq g$ by Lemma \ref{lem-2-1-5}, equivalently, $\lceil\frac{3g}{4}\rceil \leq 4$, which implies $g\leq 5$, a contradiction. Hence $x'$ is adjacent to at most three vertices of $N_1(v,C_g^\sigma)$ and the following three cases can be consider:

{\flushleft\bf Case 1.} $x'$ is adjacent to exactly three vertices $x_1,x_2$ and $x_3$ of $N_1(v,C_g^\sigma)$.

Note that $x_1, x_2, x_3$ have no common neighbors in $V(C_g^\sigma)$. We may assume that $x_i\sim y_i$ for $i\in \{1,2,3\}$. Then there exists a 3-fan of length 2 from $x'$ to $V(C_g^\sigma)$. Thus $\lfloor\frac{g}{3}\rfloor+4\geq g$ by Lemma \ref{lem-2-1-5}, and so $g=6$.

Let $x\in N_1(v,C_6^\sigma)\backslash \{x_1, x_2, x_3\}$ (if exists). Then, by Lemma \ref{lem-2-5}, $x$ is adjacent to exactly one vertex of $V(C_6^\sigma)$. If $x$ is adjacent to a vertex of $V(C_6^\sigma)$ different from $y_1, y_2$ and $y_3$, then $\Gamma[V(C_6^\sigma)\cup N_3(v,C_g^\sigma)\cup \{x_1, x_2, x_3, x, x'\}]\cong H_1^\sigma$ (see Fig.\ref{fig-2}). It is a contradiction since $i_-(\Gamma)\geq i_-(H_1^\sigma)=5>4=\lceil\frac{6}{2}\rceil+1$. Then each vertex of $N_1(v,C_g^\sigma)$ must be adjacent one of $y_1$, $y_2$ and $y_3$, and thus $N_1(v,C_g^\sigma)$ is an independent set. Therefore $\Gamma\cong \Gamma_1$ (see Fig. \ref{fig-1}), where $a, b$ and $c$ are non-negative integers.

{\flushleft\bf Case 2.} $x'$ is adjacent to exactly two vertices $x_1, x_2$ of $N_1(v,C_g^\sigma)$.

Since $x_1, x_2$ have no common neighbors in $V(C_g^\sigma)$, we may assume that $x_i\sim y_i$ for $i\in \{1,2\}$. There exists a 2-fan of length 2 from $x'$ to $V(C_g^\sigma)$. Thus $\lfloor\frac{g}{2}\rfloor+4\geq g$ by Lemma \ref{lem-2-1-5}, which implies $g=6$ or $7$.

Assume $g=6$. If $\Gamma[V(C_6^\sigma)\cup N_3(v,C_g^\sigma)\cup\{x', x_1, x_2\})]\cong H_2^\sigma$ (see Fig. \ref{fig-2}), then a contradiction $i_-(\Gamma)\geq i_-(H_2^\sigma)=5>4=\lceil\frac{6}{2}\rceil+1$ arise. Hence $\Gamma[V(C_6^\sigma)\cup N_3(v,C_g^\sigma)\cup\{x', x_1, x_2\}]\cong \Gamma_2$, where $a,b,c$ are all zero (see Fig. \ref{fig-1}). Let $x\in N_1(v,C_g^\sigma)$ (if it exists) be different from $x_1$ and $x_2$. Similar to Case 1, $x$ must be adjacent to exactly one of $y_1, y_2$ and $y_3$, and thus $N_1(v,C_g^\sigma)$ is an independent set. We get that $\Gamma$ is isomorphic to the signed graph $\Gamma_2^\sigma$, where $a, b$ and $c$ are non-negative integers.

Assume $g=7$. Then $\Gamma[V(C_7^\sigma)\cup N_3(v,C_g^\sigma)\cup\{x', x_1, x_2\})]\cong \Gamma_3$, where $a=b=c=d=e=0$ (see Fig. \ref{fig-1}). Similar to Case 1, each vertex of $N_1(v,C_g^\sigma)\setminus \{x_1, x_2\}$ (if it exists) is adjacent to exactly one of $y_1, y_2, y_3, y_4$ and $y_5$, and thus $N_1(v,C_g^\sigma)$ is an independent set. We obtain that $\Gamma$ is isomorphism to $\Gamma_3$ (see Fig. \ref{fig-1}), where $a, b, c, d$ and $e$ are non-negative integers, and $d, e$ cannot both be non-zero.

\begin{figure}[htbp]
    \centering
    \includegraphics[width=0.6\linewidth]{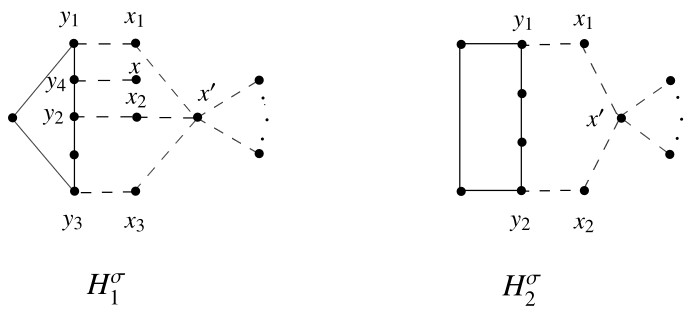}
    \caption{The signed graphs $H_1^\sigma$ and $H_2^\sigma$}
    \label{fig-2}
    \end{figure}

{\flushleft\bf Case 3.} $x'$ is adjacent to exactly one vertex $x_1$ of $N_1(v,C_g^\sigma)$.

Let $x''\in N_3(v,C_g^\sigma)$. By deleting the pendant edge $x''x'$ from $\Gamma$, we have $$\lceil\frac{g}{2}\rceil+1=i_-(\Gamma)=i_-(\Gamma[V(C_g^\sigma)\cup N_1(v,C_g^\sigma)])+1$$ due to Lemma \ref{lem-2-1-1}. Hence $i_-(\Gamma[V(C_g^\sigma)\cup N_1(v,C_g^\sigma)])=\lceil\frac{g}{2}\rceil$. Note that $\Gamma[V(C_g^\sigma)\cup N_1(v,C_g^\sigma)]$ is a canonical signed unicyclic graph by Remark \ref{re-1}. We get $\Gamma \cong K_{1, t}^\sigma(x')\odot^1\Gamma'$, where $\Gamma'$ is one of the signed graphs shown in Theorem \ref{thm-1-1}(2)(3) with balanced signed cycle $C_g^\sigma$ for $g\equiv 2,3(\bmod~4)$, and $x'$ is adjacent to exactly one pendant vertex of $\Gamma'$. We are done.
\end{proof}

\begin{thm}\label{thm-2-3}
Let $\Gamma=(G,\sigma)$ be a connected triangle-free signed graph with girth $g\equiv 2,3(\bmod~4)$. Suppose that there exists a balanced signed cycle $C_g^\sigma$ in $\Gamma$ and $\Gamma$ is not a canonical signed unicyclic graph. If $N_3(v,C_g^\sigma)=\emptyset$, then $i_-(\Gamma)=\lceil\frac{g}{2}\rceil+1$ if and only if $\Gamma$ is isomorphic to one of the following signed graphs:
\setlength{\itemsep}{0pt}
\item (1) The signed graph $(B(4,4,5),-)$. (see Fig. \ref{fig-4});
\item (2) The signed graphs $\Gamma_5$ (see Fig. \ref{fig-4}), where $a,b,c$ are non-negative integers;
\item (3)The signed graph $\Gamma_6$ (see Fig. \ref{fig-4}), where $a, b, c$ and $d$ are non-negative integers, and $d$ must be zero if any one of $a, b, c$ is non-zero;
\item (4)The signed graph $\Gamma_7$ (see Fig. \ref{fig-4}), where $a, b, c ,d$ and $e$ are non-negative integers and if $a, b, c$ are not all zero then $d=e=0$;
\item (5) The signed graph $\Gamma_8$ (see Fig. \ref{fig-4}), where $a, b, c, d$ and $e$ are non-negative integers. Moreover, at most three of $a, b, c ,d$ and $e$ are non-zero,
and $a, b, c$ are all zero whenever one of $d, e$ is non-zero;
\item (6) The signed graph $\Gamma_9$ (see Fig. \ref{fig-4}), where $a, b, c, d, e, f, g$ are non-negative integers. Moreover, at most three of them are non-zero, namely $a,b,c$. If exactly two of them are non-zero, then these two are either two from $a,b,c$, or one of $ae, ad, be, bd, fg$.
\item (7) The signed graph $K_{1, t}^\sigma(x)\odot^k\Gamma'$, where $x$ is the center vertex of $K_{1, t}^\sigma$, $\Gamma'$ is one of the signed graphs shown in Theorem \ref{thm-1-1}(1),(2) and (3) with balanced signed cycle $C_g^\sigma$ for $g\equiv 2,3(\bmod~4)$, and $x$ is adjacent to at least one vertex in $C_g^\sigma$ of $\Gamma'$.
\end{thm}
\begin{proof}
It may be assumed that each edge of $C_g^\sigma$ is positive by Theorem \ref{thm-1-0}. By Lemma \ref{lem-2-5}, any vertex in $N_1(v,C_g^\sigma)$ is adjacent to exactly one vertex of $V(C_g^\sigma)$ since $g\geq 6$. Note that $\Gamma$ is not a canonical unicyclic graph. We have $N_2(v,C_g^\sigma)\not=\emptyset$ and the following two cases can be distinguished.

{\flushleft\bf Case 1.} $N_2(v,C_g^\sigma)$ is an independent set.

Let $x_1', x_2'\in N_2(v,C_g^\sigma)$. Suppose that $N(x_1')\neq N(x_2')$. We may set $x_2'\sim x_2\in N_1(v,C_g^\sigma)$ and $x_1'\not\sim x_2$. Let $x_1'\sim x_1$. Then
$$i_-(\Gamma)\geq i_-(\Gamma[V(C_g^\sigma)\cup \{x_1,x_2,x_1',x_2'\}])=i_-(C_g^\sigma)+2>\left \lceil \frac{g}{2}\right\rceil+1$$
by Lemmas \ref{lem-2-2} and \ref{lem-2-1-1}, a contradiction. Hence $N(x_2')=N(x_1')$, which implies that $|N(x_2')|=|N(x_2')|=1$ since $g\geq 6$. Therefore, except when $|N_2(v,C_g^\sigma)|=1$, each vertex in $N_2(v,C_g^\sigma)$ has degree 1, and they share a common neighbor, say $x$, in $N_1(v,C_g^\sigma)$. Note that $\left\lceil \frac{g}{2}\right\rceil+1=i_-(\Gamma)=i_-(\Gamma-x-N_2(v,C_g^\sigma))+1$. We have $i_-(\Gamma-x-N_2(v,C_g^\sigma))=\left\lceil \frac{g}{2}\right\rceil$. Moreover, since $\Gamma-x-N_2(v,C_g^\sigma)$ is a canonical unicyclic graph by Remark \ref{re-1}, $\Gamma$ is isomorphic to the signed graph $K_{1, t}^\sigma(x)\odot^k\Gamma'$, where $\Gamma'$ is one of the signed graphs shown in Theorem \ref{thm-1-1}(1),(2) and (3) with balanced signed cycle $C_g^\sigma$ for $g\equiv 2,3(\bmod~4)$, and $x$ is adjacent to at least one vertex in $C_g^\sigma$ of $\Gamma'$.

Now we only need to consider the case that $|N_2(v,C_g^\sigma)|=1$ and the vertex in $N_2(v,C_g^\sigma)$ is adjacent to at least two vertices of $N_1(v,C_g^\sigma)$. Set $N_2(v,C_g^\sigma)=\{x'\}$. If $x'$ is adjacent to four vertices in $N_1(v,C_g^\sigma)$, then there is a $4$-fan of length 2 from $x'$ to $V(C_g^\sigma)$ since $g\geq 6$. Consequently, $\lfloor\frac{g}{4}\rfloor+4\geq g$ by Lemma \ref{lem-2-1-5}, which implies $g\leq 5$, a contradiction. Hence $x'$ is adjacent to at most three vertices in $N_1(v,C_g^\sigma)$.

{\flushleft\bf Subcase 1.} $x'$ is adjacent to exactly three vertices $x_1,x_2,$ $x_3$ in $N_1(v,C_g^\sigma)$.

We may assume that $x_i\sim y_i$ for $i\in \{1,2,3\}$ since $x_1, x_2, x_3$ have no common neighbors in $V(C_g^\sigma)$. There exists a 3-fan of length 2 from $x'$ to $V(C_g^\sigma)$. Thus $\lfloor\frac{g}{3}\rfloor+4\geq g$ by Lemma \ref{lem-2-1-5}, which implies that $g=6$. Then $\Gamma[V(C_6^\sigma)\cup \{x_1, x_2, x_3, x'\}]\cong \Gamma_4$ (see Fig.\ref{fig-4}). In the order of $y_1, y_2, y_3, y_4, y_5, y_6, x_1, x_2, x_3, x'$, $A(\Gamma_4)$ can be written as
$$A(\Gamma_4)=\left(
        \begin{array}{cccccccccc}
          0 & 0 & 0 & 1 & 1 & 0 & a_{17}^\sigma & 0 & 0 & 0\\
          0 & 0 & 0 & 0 & 1 & 1 & 0 & a_{28}^\sigma & 0 & 0\\
          0 & 0 & 0 & 1 & 0 & 1 & 0 & 0 & a_{39}^\sigma & 0\\
          1 & 0 & 1 & 0 & 0 & 0 & 0 & 0 & 0 & 0\\
          1 & 1 & 0 & 0 & 0 & 0 & 0 & 0 & 0 & 0\\
          0 & 1 & 1 & 0 & 0 & 0 & 0 & 0 & 0 & 0\\
          a_{17}^\sigma & 0 & 0 & 0 & 0 & 0 & 0 & 0 & 0 & a_{710}^\sigma\\
          0 & a_{28}^\sigma & 0 & 0 & 0 & 0 & 0 & 0 & 0 & a_{810}^\sigma\\
          0 & 0 & a_{39}^\sigma & 0 & 0 & 0 & 0 & 0 & 0 & a_{910}^\sigma\\
          0 & 0 & 0 & 0 & 0 & 0 & a_{710}^\sigma & a_{810}^\sigma & a_{910}^\sigma & 0\\
        \end{array}
      \right).$$
Applying elementary congruence matrix operations to $A(\Gamma_4)$, it is congruent to

$$B_1=\left(
        \begin{array}{cccccccccccc}
          0 & 1 & 0 & 0 & 0 & 0 & 0 & 0 & 0 & 0\\
          1 & 0 & 0 & 0 & 0 & 0 & 0 & 0 & 0 & 0\\
          0 & 0 & 0 & 1 & 0 & 0 & 0 & 0 & 0 & 0\\
          0 & 0 & 1 & 0 & 0 & 0 & 0 & 0 & 0 & 0\\
          0 & 0 & 0 & 0 & 0 & 2 & 0 & 0 & 0 & 0\\
          0 & 0 & 0 & 0 & 2 & 0 & 0 & 0 & 0 & 0\\
          0 & 0 & 0 & 0 & 0 & 0 & 0 & 0 & 0 & 0\\
          0 & 0 & 0 & 0 & 0 & 0 & 0 & 0 & 0 & 0\\
          0 & 0 & 0 & 0 & 0 & 0 & 0 & 0 & 0 & a_{910}^\sigma\\
          0 & 0 & 0 & 0 & 0 & 0 & 0 & 0& a_{910}^\sigma & 0\\
        \end{array}
      \right).$$
The matrix $B_1$ has four negative eigenvalues. Let $x\in N_1(v,C_6^\sigma)$ $\backslash \{x_1, x_2, x_3\}$ (if it exists). Then, by Lemma \ref{lem-2-5}, $x$ is adjacent to exactly one vertex of $V(C_6^\sigma)$. Suppose that $x$ is adjacent to a vertex of $V(C_6^\sigma)$ other than $y_1, y_2$ and $y_3$. We find $\Gamma[V(C_6^\sigma)\cup \{x_1, x_2, x_3, x, x'\}]\cong H_3^\sigma$ (see Fig.\ref{fig-3}). In the order of $y_1, y_2, y_3, y_4, y_5, y_6, x_1, x_2, x_3, x, x'$, $A(H_3^\sigma)$ can be written as
$$A(H_3^\sigma)=\left(
        \begin{array}{ccccccccccc}
          0 & 0 & 0 & 1 & 0 & 1 & a_{17}^\sigma & 0 & 0 & 0 & 0\\
          0 & 0 & 0 & 1 & 1 & 0 & 0 & a_{28}^\sigma & 0 & 0 & 0\\
          0 & 0 & 0 & 0 & 1 & 1 & 0 & 0 & a_{39}^\sigma & 0 & 0\\
          1 & 1 & 0 & 0 & 0 & 0 & 0 & 0 & 0 & a_{410}^\sigma & 0\\
          0 & 1 & 1 & 0 & 0 & 0 & 0 & 0 & 0 & 0 & 0\\
          1 & 0 & 1 & 0 & 0 & 0 & 0 & 0 & 0 & 0 & 0\\
          a_{17}^\sigma & 0 & 0 & 0 & 0 & 0 & 0 & 0 & 0 & 0 & a_{711}^\sigma\\
          0 & a_{28}^\sigma & 0 & 0 & 0 & 0 & 0 & 0 & 0 & 0 & a_{811}^\sigma\\
          0 & 0 & a_{39}^\sigma & 0 & 0 & 0 & 0 & 0 & 0 & 0 & a_{911}^\sigma\\
          0 & 0 & 0 & a_{410}^\sigma & 0 & 0 & 0 & 0 & 0 & 0 & 0\\
          0 & 0 & 0 & 0 & 0 & 0 & a_{711}^\sigma & a_{811}^\sigma & a_{911}^\sigma & 0 & 0\\
        \end{array}
      \right).$$
Applying elementary congruence matrix operations to $A(H_3^\sigma)$, it is congruent to
$$B_2=\left(
        \begin{array}{ccccccccccccc}
          0 & 1 & 0 & 0 & 0 & 0 & 0 & 0 & 0 & 0 & 0\\
          1 & 0 & 0 & 0 & 0 & 0 & 0 & 0 & 0 & 0 & 0\\
          0 & 0 & 0 & 2 & 0 & 0 & 0 & 0 & 0 & 0 & 0\\
          0 & 0 & 2 & 0 & 0 & 0 & 0 & 0 & 0 & 0 & 0\\
          0 & 0 & 0 & 0 & 0 & 1 & 0 & 0 & 0 & 0 & 0\\
          0 & 0 & 0 & 0 & 1 & 0 & 0 & 0 & 0 & 0 & 0\\
          0 & 0 & 0 & 0 & 0 & 0 & 0 & a_{811}^\sigma+\frac{a_{28}^\sigma a_{911}^\sigma}{a_{39}^\sigma} & 0 & 0 & 0\\
          0 & 0 & 0 & 0 & 0 & 0 & a_{811}^\sigma+\frac{a_{28}^\sigma a_{911}^\sigma}{a_{39}^\sigma} & 0 & 0 & 0 & 0\\
          0 & 0 & 0 & 0 & 0 & 0 & 0 & 0 & 0 & 0 & 0\\
          0 & 0 & 0 & 0 & 0 & 0 & 0 & 0 & 0 & 0 & \frac{a_{39}^\sigma a_{410}^\sigma}{2}\\
          0 & 0 & 0 & 0 & 0 & 0 & 0 & 0 & 0 & \frac{a_{39}^\sigma a_{410}^\sigma}{2} & 0\\
        \end{array}
      \right).$$
Clearly, $B_2$ has five negative eigenvalues. A contradiction $i_-(\Gamma)\geq i_-(\Gamma[V(C_6^\sigma)\cup \{x_2,x_3,x,\\ x'\}])=5>4=\lceil\frac{6}{2}\rceil+1$ arises. Then each vertex of $N_1(v,C_g^\sigma)$ must be adjacent to exactly one of $y_1$, $y_2$ and $y_3$, and thus $N_1(v,C_g^\sigma)$ is an independent set. Therefore $\Gamma\cong \Gamma_5$ (see Fig. \ref{fig-4}), where $a, b$ and $c$ are non-negative integers.

{\flushleft\bf Subcase 2.} $x'$ is adjacent to exactly two vertices $x_1, x_2$ of $N_1(v,C_g^\sigma)$.

We may assume that $x_i\sim y_i$ for $i\in \{1,2\}$. There exists a 2-fan of length 2 from $x'$ to $V(C_g^\sigma)$. Thus $\lfloor\frac{g}{2}\rfloor+4\geq g$ by Lemma \ref{lem-2-1-5}, and so $g=6$ or $7$.

Assume $g=6$ and $N_1(v,C_6^\sigma)=\{x_1, x_2\}$. Since $i_-((B(4,4,5),+))=5>4=\lceil\frac{6}{2}\rceil+1$, $i_-((B(5,3,5),+))=i_-(B(5,3,5),-)=i_-(B(4,4,5),-)=4=\lceil\frac{6}{2}\rceil+1$ (see Fig. \ref{fig-4} and \ref{fig-3}). $\Gamma$ is isomorphic to the signed graph $(B(5,3,5),+)$, $(B(5,3,5),-)$ or $(B(4,4,5),-)$ by Lemma \ref{lem-2-11}. A simple observation shows that $N_1(v,C_6^\sigma)=\{x_1, x_2\}$ if $\Gamma[V(C_6^\sigma)\cup \{x_1, x_2\}\cup x']\cong (B(4,4,5),-)$. We consider the other two graphs and let $x\in N_1(v,C_6^\sigma)$ be different from $x_1$ and $x_2$. Similar to Subcase 1, $x$ must be adjacent to exactly one of $y_1, y_2$, $y_3$, $y_5$ in $(B(5,3,5),+)$. Hence $\Gamma$ is isomorphic to the signed graph $\Gamma_6$, where $a, b, c$ and $d$ are non-negative integers (see Fig. \ref{fig-4}). Moreover, if any one of $a, b, c$ is non-zero, then $d$ must be zero. Similarly, $x$ must be adjacent to exactly one of $y_1, y_2$, $y_4$, $y_5$, $y_6$ in $(B(5,3,5),-)$. Hence $\Gamma$ is isomorphic to the signed graph $\Gamma_7$, where $a, b, c ,d$ and $e$ are non-negative integers (see Fig. \ref{fig-4}). Moreover, at most three of $a, b, c ,d$ and $e$ are non-zero, and $a, b, c$ are all zero whenever one of $d, e$ is non-zero.

Assume $g=7$ and $N_1(v,C_7^\sigma)=\{x_1, x_2\}$. Then by Lemma \ref{lem-2-11}, $\Gamma$ is switching equivalent to $(B(5,4,5),+)$ or $(B(5,4,5),-)$ (see Fig. \ref{fig-4}). By a simple operation, we get $$i_-((B(5,4,5),+))=i_-((B(5,4,5),-))=5=\lceil\frac{7}{2}\rceil+1,$$ as desired.
If $N_1(v,C_7^\sigma)\neq \{x_1, x_2\}$ and set $x\in N_1(v,C_7^\sigma)\setminus \{x_1, x_2\}$, then $x$ must be adjacent to exactly one of $y_1, y_2$, $y_3, y_4$ and $y_5$ in $(B(5,4,5),+)$. Hence $\Gamma$ is isomorphic to the signed graph $\Gamma_8$, where $a, b, c, d$ and $e$ are non-negative integers (see Fig. \ref{fig-4}). Moreover, at most four of them are non-zero, and $a, b, c$ are all zero whenever $d$ and $e$ are both non-zero. Similarly, $x$ must be adjacent to exactly one of $y_1, y_2$, $y_3, y_4$, $y_5$, $y_6$ and $y_7$ in $(B(5,4,5),-)$. Then $\Gamma$ is isomorphic to the signed graph $\Gamma_9$, where $a, b, c, d, e, f, g$ are non-negative integers (see Fig. \ref{fig-4}). Moreover, at most three of them are non-zero, namely $a,b,c$. If exactly two of them are non-zero, then these two are either two from $a,b,c$, or one of $ae, ad, be, bd, fg$.

\begin{figure}[ht]
    \centering
    \includegraphics[width=0.95\linewidth]{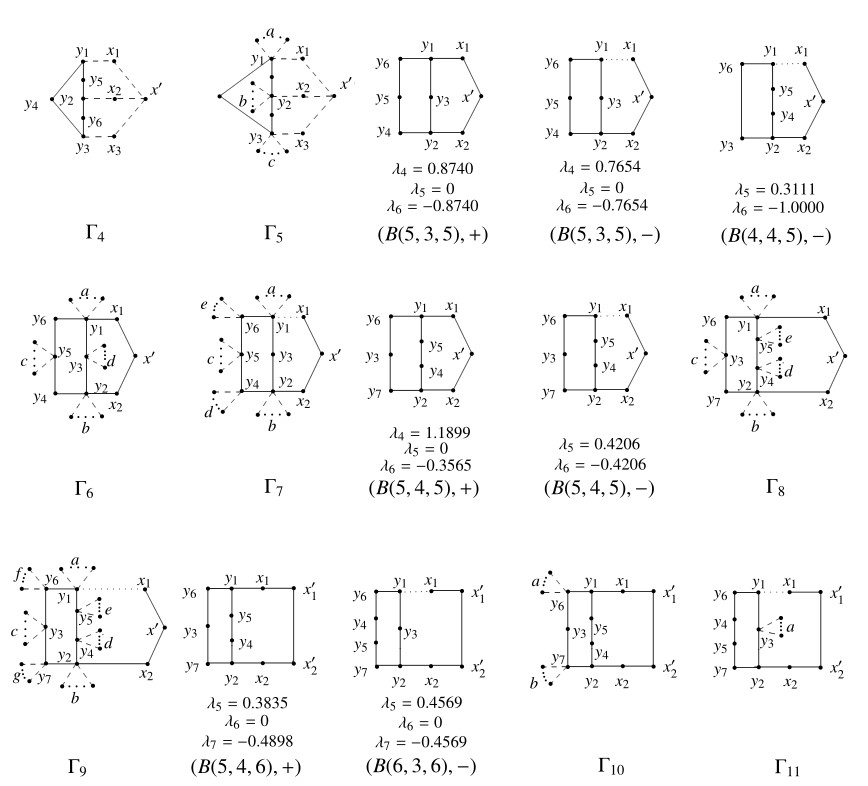}
    \caption{The signed graphs $\Gamma_4$, $\Gamma_5$, $(B(5,3,5),+)$, $(B(5,3,5),-)$, $(B(4,4,5),-)$, $\Gamma_6$, $\Gamma_7$, $(B(5,4,5),+)$, $(B(5,4,5),-)$, $\Gamma_4$, $\Gamma_8$, $\Gamma_9$, $(B(5,4,6),+)$, $(B(6,3,6),{\sigma_1})$, $\Gamma_{10}$ and  $\Gamma_{11}$}
    \label{fig-4}
    \end{figure}

{\flushleft\bf Case 2.} $N_2(v,C_g^\sigma)$ is not an independent set.

We may assume that $x_1', x_2'\in N_2(v,C_g^\sigma)$ and $x_1'\sim x_2'$ since $N_2(v,C_g^\sigma)$ is not an independent set. Let $x_i'\sim x_i\in N_1(v,C_g^\sigma)$ and $x_i\sim y_i$ for $i\in\{1,2\}$. Firstly, the following claim can be established.
\begin{claim}
$N_2(v, C_g^\sigma)]=\{x_1', x_2'\}$.
\end{claim}
\begin{proof}
Assume that $x_3'\in N_2(v, C_g^\sigma)$ besides $x_1'$ and $x_2'$. Then $x_3'$ cannot be adjacent to both $x_1$ and $x_2$ because $g\geq 6$. If $x_3'$ is adjacent to a vertex $x_3\in N_1(v,C_g^\sigma)\backslash\{x_1, x_2\}$, then
$$i_-(\Gamma)\geq i_-(\Gamma[V(C_g^\sigma)\cup\{x_1,x_3,x_1',x_3'\}])=i_-(\Gamma[V(C_g^\sigma)])+2=\lceil\frac{g}{2}\rceil+2>\left\lceil\frac{g}{2}\right\rceil+1.$$
If $x_3'$ is adjacent one of $x_1$ and $x_2$, say $x_2$, then
$$i_-(\Gamma)\geq i_-(\Gamma[V(C_g^\sigma)\cup\{x_1, x_2, x_1', x_3'\}])=i_-(\Gamma[V(C_g^\sigma)])+2=\lceil\frac{g}{2}\rceil+2>\left\lceil\frac{g}{2}\right\rceil+1.$$
 These two contradictions show $N_2(v, C_g^\sigma)=\{x'_1, x'_2\}$.
\end{proof}

\begin{figure}[ht]
    \centering
    \includegraphics[width=0.95\linewidth]{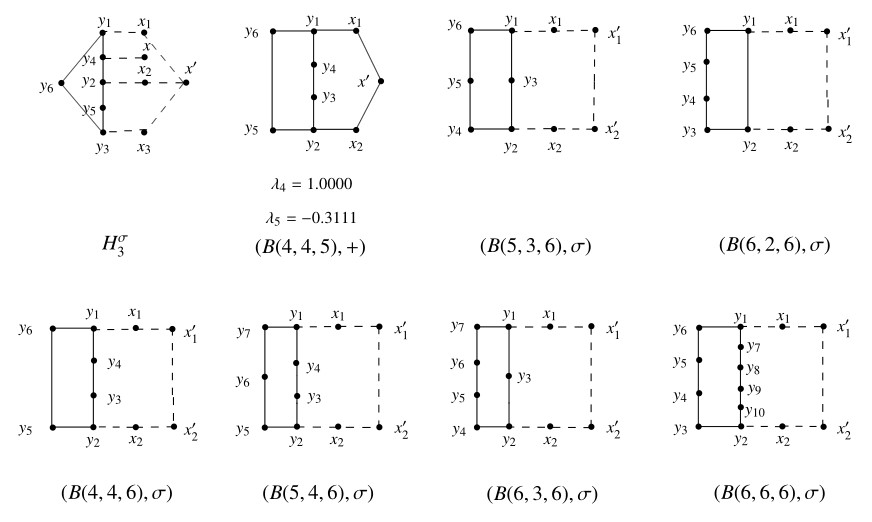}
    \caption{The signed graphs $H_3^\sigma$, $(B(4,4,5),+)$, $(B(5,3,6),\sigma)$, $(B(6,2,6),\sigma)$, $(B(4,4,6),\sigma)$, $(B(5,4,6),\sigma)$, $(B(6,3,6),\sigma)$ and $(B(6,6,6),\sigma)$ }
    \label{fig-3}
    \end{figure}

Since there exists a path of length 5 between $y_1$ and $y_2$ such that $P$ is vertex-disjoint from $C_g^\sigma$, we get $\left \lfloor \frac{g}{2}\right\rfloor+5\geq g$ by Lemma \ref{lem-2-5}, and so $g\leq 10$. Hence $g=6,7$ or $10$.

Assume $g=6$. Then $\Gamma[V(C_6^\sigma)\cup\{x_1,x_2,x_1',x_2'\}]\cong (B(5,3,6), \sigma)$, $(B(6,2,6),\sigma)$ or $ (B(4,4,6),\sigma)$ (see Fig. \ref{fig-3}). In order of $y_1, y_2, y_3, y_4, y_5, y_6, x_1, x_2, x_1', x_2'$, $A((B(5,3,6),\sigma))$ can be written as
$$A((B(5,3,6),\sigma))=\left(
        \begin{array}{cccccccccc}
          0 & 0 & 1 & 0 & 0 & 1 & a_{17}^\sigma & 0 & 0 & 0\\
          0 & 0 & 1 & 1 & 0 & 0 & 0 & a_{28}^\sigma & 0 & 0\\
          1 & 1 & 0 & 0 & 0 & 0 & 0 & 0 & 0 & 0\\
          0 & 1 & 0 & 0 & 1 & 0 & 0 & 0 & 0 & 0\\
          0 & 0 & 0 & 1 & 0 & 1 & 0 & 0 & 0 & 0\\
          1 & 0 & 0 & 0 & 1 & 0 & 0 & 0 & 0 & 0\\
          a_{17}^\sigma & 0 & 0 & 0 & 0 & 0 & 0 & 0 & a_{79}^\sigma & 0\\
          0 & a_{28}^\sigma & 0 & 0 & 0 & 0 & 0 & 0 & 0 & a_{810}^\sigma\\
          0 & 0 & 0 & 0 & 0 & 0 & a_{79}^\sigma & 0 & 0 & a_{910}^\sigma\\
          0 & 0 & 0 & 0 & 0 & 0 & 0 & a_{810}^\sigma & a_{910}^\sigma & 0\\
        \end{array}
      \right).$$
Applying elementary congruence matrix operations to $A((B(5,3,6),\sigma))$, it is congruent to
$$B_3=\left(
        \begin{array}{cccccccccccc}
          0 & 1 & 0 & 0 & 0 & 0 & 0 & 0 & 0 & 0\\
          1 & 0 & 0 & 0 & 0 & 0 & 0 & 0 & 0 & 0\\
          0 & 0 & 0 & 1 & 0 & 0 & 0 & 0 & 0 & 0\\
          0 & 0 & 1 & 0 & 0 & 0 & 0 & 0 & 0 & 0\\
          0 & 0 & 0 & 0 & 0 & 2 & 0 & 0 & 0 & 0\\
          0 & 0 & 0 & 0 & 2 & 0 & 0 & 0 & 0 & 0\\
          0 & 0 & 0 & 0 & 0 & 0 & 0 & a_{810}^\sigma & 0 & 0\\
          0 & 0 & 0 & 0 & 0 & 0 & a_{810}^\sigma & 0 & 0 & 0\\
          0 & 0 & 0 & 0 & 0 & 0 & 0 & 0 & 0 & a_{79}^\sigma\\
          0 & 0 & 0 & 0 & 0 & 0 & 0 & 0& a_{79}^\sigma & 0\\
        \end{array}
      \right).$$

The matrix $B_3$ has five negative eigenvalues, a contradiction $i_-((B(5,3,6),\sigma))=5>\lceil\frac{6}{2}\rceil+1=4$ follows from Lemma \ref{lem-2-0}.

In order of $y_1, y_2, y_3, y_4, y_5, y_6, x_1, x_2, x_1', x_2'$, $A((B(6,2,6),\sigma))$ can be written as
$$A((B(6,2,6),\sigma))=\left(
        \begin{array}{cccccccccc}
          0 & 1 & 0 & 0 & 0 & 1 & a_{17}^\sigma & 0 & 0 & 0\\
          1 & 0 & 1 & 0 & 0 & 0 & 0 & a_{28}^\sigma & 0 & 0\\
          0 & 1 & 0 & 1 & 0 & 0 & 0 & 0 & 0 & 0\\
          0 & 0 & 1 & 0 & 1 & 0 & 0 & 0 & 0 & 0\\
          0 & 0 & 0 & 1 & 0 & 1 & 0 & 0 & 0 & 0\\
          1 & 0 & 0 & 0 & 1 & 0 & 0 & 0 & 0 & 0\\
          a_{17}^\sigma & 0 & 0 & 0 & 0 & 0 & 0 & 0 & a_{79}^\sigma & 0\\
          0 & a_{28}^\sigma & 0 & 0 & 0 & 0 & 0 & 0 & 0 & a_{810}^\sigma\\
          0 & 0 & 0 & 0 & 0 & 0 & a_{79}^\sigma & 0 & 0 & a_{910}^\sigma\\
          0 & 0 & 0 & 0 & 0 & 0 & 0 & a_{810}^\sigma & a_{910}^\sigma & 0\\
        \end{array}
      \right).$$
Applying elementary congruence matrix operations to $A((B(6,2,6),\sigma))$, it is congruent to
$$
\setlength{\arraycolsep}{9pt}B_4=\left(
        \begin{array}{cccccccccccc}
          0 & 1 & 0 & 0 & 0 & 0 & 0 & 0 & 0 & 0\\
          1 & 0 & 0 & 0 & 0 & 0 & 0 & 0 & 0 & 0\\
          0 & 0 & 0 & 1 & 0 & 0 & 0 & 0 & 0 & 0\\
          0 & 0 & 1 & 0 & 0 & 0 & 0 & 0 & 0 & 0\\
          0 & 0 & 0 & 0 & 0 & 2 & 0 & 0 & 0 & 0\\
          0 & 0 & 0 & 0 & 2 & 0 & 0 & 0 & 0 & 0\\
          0 & 0 & 0 & 0 & 0 & 0 & 0 & \frac{-a_{17}^\sigma a_{28}^\sigma}{2} & 0 & 0\\
          0 & 0 & 0 & 0 & 0 & 0 & \frac{-a_{17}^\sigma a_{28}^\sigma}{2} & 0 & 0 & 0\\
          0 & 0 & 0 & 0 & 0 & 0 & 0 & 0 & 0 & a_{910}^\sigma+\frac{2a_{79}^\sigma a_{810}^\sigma}{a_{17}^\sigma a_{28}^\sigma}\\
          0 & 0 & 0 & 0 & 0 & 0 & 0 & 0& a_{910}^\sigma+\frac{2a_{79}^\sigma a_{810}^\sigma}{a_{17}^\sigma a_{28}^\sigma} & 0\\
        \end{array}
      \right).$$
Since $\frac{2a_{79}^\sigma a_{810}^\sigma}{a_{17}^\sigma a_{28}^\sigma}\not=1,-1$, the matrix $B_4$ has five negative eigenvalues, a contradiction $i_-((B\\(6,2,6),\sigma))=5>\lceil\frac{6}{2}\rceil+1=4$ arises by Lemma \ref{lem-2-0}.

At last,
%$A((B(4,4,6),\sigma))$
%can be written as
%$$A((B(4,4,6),\sigma))=\left(
       % \begin{array}{cccccccccc}
          %0 & 1 & 0 & 0 & 0 & 1 & 0 & 0 & 0 & 0\\
         % 1 & 0 & 1 & 0 & 0 & 0 & 0 & 0 & 0 & 0\\
         % 0 & 1 & 0 & 1 & 0 & 0 & 0 & 0 & 0 & a_{310}^\sigma\\
          %0 & 0 & 1 & 0 & 1 & 0 & 0 & 0 & 0 & 0\\
          %0 & 0 & 0 & 1 & 0 & 1 & 0 & 0 & 0 & 0\\
          %1 & 0 & 0 & 0 & 1 & 0 & a_{67}^\sigma & 0 & 0 & 0\\
          %0 & 0 & 0 & 0 & 0 & a_{67}^\sigma & 0 & a_{78}^\sigma & 0 & 0\\
         % 0 & 0 & 0 & 0 & 0 & 0 & a_{78}^\sigma & 0 & a_{89}^\sigma & 0\\
         % 0 & 0 & 0 & 0 & 0 & 0 & 0 & a_{89}^\sigma & 0 & a_{910}^\sigma\\
          %0 & 0 & a_{310}^\sigma & 0 & 0 & 0 & 0 & 0 & a_{910}^\sigma & 0\\
        %\end{array}
      %\right).$$
%Applying elementary congruence matrix operations on $A((B(4,4,6),\sigma))$, we get that
note that $A((B(4,4,6),\sigma))$ is congruent to
$$
\setlength{\arraycolsep}{10pt}B_5=\left(
        \begin{array}{cccccccccccc}
          0 & 1 & 0 & 0 & 0 & 0 & 0 & 0 & 0 & 0\\
          1 & 0 & 0 & 0 & 0 & 0 & 0 & 0 & 0 & 0\\
          0 & 0 & 0 & 1 & 0 & 0 & 0 & 0 & 0 & 0\\
          0 & 0 & 1 & 0 & 0 & 0 & 0 & 0 & 0 & 0\\
          0 & 0 & 0 & 0 & 0 & 2 & 0 & 0 & 0 & 0\\
          0 & 0 & 0 & 0 & 2 & 0 & 0 & 0 & 0 & 0\\
          0 & 0 & 0 & 0 & 0 & 0 & 0 & \frac{a_{17}^\sigma a_{28}^\sigma}{2} & 0 & 0\\
          0 & 0 & 0 & 0 & 0 & 0 & \frac{a_{17}^\sigma a_{28}^\sigma}{2} & 0 & 0 & 0\\
          0 & 0 & 0 & 0 & 0 & 0 & 0 & 0 & 0 & a_{910}^\sigma-\frac{2a_{79}^\sigma a_{810}^\sigma}{a_{17}^\sigma a_{28}^\sigma}\\
          0 & 0 & 0 & 0 & 0 & 0 & 0 & 0 & a_{910}^\sigma-\frac{2a_{79}^\sigma a_{810}^\sigma}{a_{17}^\sigma a_{28}^\sigma} & 0\\
        \end{array}
      \right).$$
Since $\frac{2a_{79}^\sigma a_{810}^\sigma}{a_{17}^\sigma a_{28}^\sigma}\not=1,-1$, the matrix $B_5$ has five negative eigenvalues, a contradiction $i_-((B$ $(4,4,6),\sigma))=5>\lceil\frac{6}{2}\rceil+1=4$ arises by Lemma \ref{lem-2-0}.
%\textcolor[rgb]{1.00,0.00,0.00}{
%If $N_1(v,C_6^\sigma)\neq \{x_1, x_2\}$ and set $x\in N_1(v,C_6^\sigma)\setminus \{x_1, x_2\}$, then $x$ must be adjacent to exactly both $y_6$ and $y_4$ in $(B(5,3,6),{\sigma_1})$ (see Fig. \ref{fig-3}), Hence $\Gamma$ is isomorphic to the signed graph $\Gamma_{10}$ (see Fig. \ref{fig-4}), where $a, b$ are non-negative integers and $a,b$ are not all zero. Similarly, $x$ must be adjacent only to $y_3$ in $(B(5,3,6),{\sigma_2})$ (see Fig. \ref{fig-3}). Then $\Gamma$ is isomorphic to the signed graph $\Gamma_{11}$ (see Fig. \ref{fig-4}). At last, $x$ cannot be adjacent to any one of $y_i$ for integer $1\leq i\leq 6$ in $(B(6,2,6),\sigma)$. Similarly, $x$ cannot be adjacent to any one of $y_i$ for integer $1\leq i\leq 6$ in $(B(4,4,6),\sigma)$.}

Assume $g=7$. Then $\Gamma[V(C_7^\sigma)\cup\{x_1,x_2,x_1',x_2'\}]\cong(B(5,4,6),\sigma)$ or $(B(6,3,6),\sigma)$ (see Fig. \ref{fig-3}). In order of $y_1, y_2, y_3, y_4, y_5, y_6, y_7, x_1, x_2, x_1', x_2'$, it can be written as
$$A((B(5,4,6),\sigma))=\left(
          \begin{array}{ccccccccccc}
          0 & 0 & 0 & 1 & 0 & 0 & 1 & a_{18}^\sigma & 0 & 0 & 0\\
          0 & 0 & 1 & 0 & 1 & 0 & 0 & 0 & a_{29}^\sigma & 0 & 0\\
          0 & 1 & 0 & 1 & 0 & 0 & 0 & 0 & 0 & 0 & 0\\
          1 & 0 & 1 & 0 & 0 & 0 & 0 & 0 & 0 & 0 & 0\\
          0 & 1 & 0 & 0 & 0 & 1 & 0 & 0 & 0 & 0 & 0\\
          0 & 0 & 0 & 0 & 1 & 0 & 1 & 0 & 0 & 0 & 0\\
          1 & 0 & 0 & 0 & 0 & 1 & 0 & 0 & 0 & 0 & 0\\
          a_{18}^\sigma & 0 & 0 & 0 & 0 & 0 & 0 & 0 & 0 & a_{810}^\sigma & 0\\
          0 & a_{29}^\sigma & 0 & 0 & 0 & 0 & 0 & 0 & 0 & 0 & a_{911}^\sigma\\
          0 & 0 & 0 & 0 & 0 & 0 & 0 & a_{811}^\sigma & 0 & 0 & a_{1011}^\sigma\\
          0 & 0 & 0 & 0 & 0 & 0 & 0 & 0 & a_{911}^\sigma & a_{1011}^\sigma & 0\\
        \end{array}
      \right).$$

Applying elementary congruence matrix operations to $A((B(5,4,6),\sigma))$, it is congruent to
$$
\setlength{\arraycolsep}{12pt}
\left(
        \begin{array}{ccccccccccccc}
          0 & 1 & 0 & 0 & 0 & 0 & 0 & 0 & 0 & 0 & 0\\
          1 & 0 & 0 & 0 & 0 & 0 & 0 & 0 & 0 & 0 & 0\\
          0 & 0 & 0 & 1 & 0 & 0 & 0 & 0 & 0 & 0 & 0\\
          0 & 0 & 1 & 0 & 0 & 0 & 0 & 0 & 0 & 0 & 0\\
          0 & 0 & 0 & 0 & 0 & 1 & 0 & 0 & 0 & 0 & 0\\
          0 & 0 & 0 & 0 & 1 & 0 & 0 & 0 & 0 & 0 & 0\\
          0 & 0 & 0 & 0 & 0 & 0 & -2 & 0 & 0 & 0 & 0\\
          0 & 0 & 0 & 0 & 0 & 0 & 0 & \frac{{a_{18}^\sigma}^2}{2} & 0 & 0 & 0\\
          0 & 0 & 0 & 0 & 0 & 0 & 0 & 0 & 0 & \frac{-{a_{29}^\sigma}{a_{810}^\sigma}}{{a_{18}^\sigma}} & 0\\
          0 & 0 & 0 & 0 & 0 & 0 & 0 & 0 & \frac{-{a_{29}^\sigma}{a_{810}^\sigma}}{{a_{18}^\sigma}} & 0 & 0 \\
          0 & 0 & 0 & 0 & 0 & 0 & 0 & 0 & 0 & 0 & \lambda_1\\
        \end{array}
      \right),$$
where $\lambda_1=\frac{2{a_{18}^\sigma}{a_{911}^\sigma}{a_{1011}^\sigma}}{{a_{29}^\sigma}{a_{810}^\sigma}}$. Then $i_-(B(5,4,6),\sigma)=5=\lceil \frac{7}{2}\rceil+1$ if and only if $a_{18}^\sigma a_{911}^\sigma a_{1011}^\sigma$ and $a_{29}^\sigma a_{810}^\sigma$ have the same sign. All $16$ signed graphs satisfying the same sign condition are switching equivalent to $(B(5,4,6),+)$ (see Fig. \ref{fig-4}). Let $x\in N_1(v,C_7^\sigma)\setminus \{x_1, x_2\}$ (if it exists). Then $x$ must be adjacent to exactly one of $y_6$ and $y_7$ in $(B(5,4,6),+)$. Hence $N_1(v,C_7^\sigma)\setminus \{x_1, x_2\}$ is an independent set. We get $\Gamma\cong \Gamma_{10}$ (see Fig. \ref{fig-4}), where $a$ and $b$ are non-negative integers. % 未改说明:平衡8 圈的负惯性指数为3

%In order of $y_1, y_2, y_3, y_4, y_5, y_6, y_7, x_1, x_2, x_1', x_2' $, $A((B(6,3,6),\sigma))$ can be written as
%$$A((B(6,3,6),\sigma))=\left(
        %\begin{array}{ccccccccccc}
          %0 & 0 & 1 & 0 & 0 & 0 & 1 & a_{18}^\sigma & 0 & 0 & 0\\
          %0 & 0 & 1 & 1 & 0 & 0 & 0 & 0 & a_{29}^\sigma & 0 & 0\\
          %1 & 1 & 0 & 0 & 0 & 0 & 0 & 0 & 0 & 0 & 0\\
          %0 & 1 & 0 & 0 & 1 & 0 & 0 & 0 & 0 & 0 & 0\\
          %0 & 0 & 0 & 1 & 0 & 1 & 0 & 0 & 0 & 0 & 0\\
          %0 & 0 & 0 & 0 & 1 & 0 & 1 & 0 & 0 & 0 & 0\\
          %1 & 0 & 0 & 0 & 0 & 1 & 0 & 0 & 0 & 0 & 0\\
          %a_{18}^\sigma & 0 & 0 & 0 & 0 & 0 & 0 & 0 & 0 & a_{810}^\sigma & 0\\
          %0 & a_{29}^\sigma & 0 & 0 & 0 & 0 & 0 & 0 & 0 & 0 & a_{911}^\sigma\\
          %0 & 0 & 0 & 0 & 0 & 0 & 0 & a_{810}^\sigma & 0 & 0 & a_{1011}^\sigma\\
          %0 & 0 & 0 & 0 & 0 & 0 & 0 & 0 & a_{911}^\sigma & a_{1011}^\sigma & 0\\
        %\end{array}
      %\right).$$
Similarly, applying elementary congruence matrix operations to $A((B(6,3,6),\sigma))$, it is congruent to
$$
\setlength{\arraycolsep}{12pt}\left(
        \begin{array}{ccccccccccccc}
          0 & -1 & 0 & 0 & 0 & 0 & 0 & 0 & 0 & 0 & 0\\
          -1 & 0 & 0 & 0 & 0 & 0 & 0 & 0 & 0 & 0 & 0\\
          0 & 0 & 0 & 1 & 0 & 0 & 0 & 0 & 0 & 0 & 0\\
          0 & 0 & 1 & 0 & 0 & 0 & 0 & 0 & 0 & 0 & 0\\
          0 & 0 & 0 & 0 & 0 & 1 & 0 & 0 & 0 & 0 & 0\\
          0 & 0 & 0 & 0 & 1 & 0 & 0 & 0 & 0 & 0 & 0\\
          0 & 0 & 0 & 0 & 0 & 0 & -2 & 0 & 0 & 0 & 0\\
          0 & 0 & 0 & 0 & 0 & 0 & 0 & \frac{{a_{18}^\sigma}^2}{2} & 0 & 0 & 0\\
          0 & 0 & 0 & 0 & 0 & 0 & 0 & 0 & 0 & \frac{{a_{29}^\sigma}{a_{810}^\sigma}}{{a_{18}^\sigma}} & 0\\
          0 & 0 & 0 & 0 & 0 & 0 & 0 & 0 & \frac{{a_{29}^\sigma}{a_{810}^\sigma}}{{a_{18}^\sigma}} & 0 & 0 \\
          0 & 0 & 0 & 0 & 0 & 0 & 0 & 0 & 0 & 0 & \lambda_2\\
        \end{array}
      \right),$$
where $\lambda_2=\frac{-2{a_{18}^\sigma}{a_{911}^\sigma}{a_{1011}^\sigma}}{{a_{29}^\sigma}{a_{810}^\sigma}}$. Then $i_-((B(6,3,6),\sigma))=5=\lceil \frac{7}{2}\rceil+1$ if and only if $a_{18}^\sigma a_{911}^\sigma a_{1011}^\sigma$ and $a_{29}^\sigma a_{810}^\sigma$ have opposite signs. All $16$ signed graphs satisfying the opposite sign condition are switching equivalent to $(B(6,3,6),-)$. Let $x\in N_1(v,C_7^\sigma)\setminus \{x_1, x_2\}$ (if it exists). Then $x$ is only adjacent to $y_3$ in $(B(6,3,6),-)$. Hence $N_1(v,C_7^\sigma)\setminus \{x_1, x_2\}$ is an independent set. We get $\Gamma\cong \Gamma_{11}$ (see Fig. \ref{fig-4}), where $a$ is a non-negative integers.

Assume $g=10$. Then $\Gamma[V(C_{10}^\sigma)\cup\{x_1,x_2,x_1',x_2'\}]\cong (B(6,6,6),\sigma)$ (see Fig. \ref{fig-3}). In order of $y_1, y_2, y_3, y_4, y_5, y_6, y_7, y_8$, $ y_9, y_{10},x_1, x_2, x_1',x_2' $, $A((B(6,6,6),\sigma))$ can be written as\\
$$A((B(6,6,6),\sigma))=\left(
        \begin{array}{cccccccccccccc}
          0 & 0 & 0 & 0 & 0 & 1 & 1 & 0 & 0 & 0 & a_{111}^\sigma & 0 & 0 & 0\\
          0 & 0 & 1 & 0 & 0 & 0 & 0 & 0 & 0 & 1 & 0 & a_{212}^\sigma & 0 & 0\\
          0 & 1 & 0 & 1 & 0 & 0 & 0 & 0 & 0 & 0 & 0 & 0 & 0 & 0\\
          0 & 0 & 1 & 0 & 1 & 0 & 0 & 0 & 0 & 0 & 0 & 0 & 0 & 0\\
          0 & 0 & 0 & 1 & 0 & 1 & 0 & 0 & 0 & 0 & 0 & 0 & 0 & 0\\
          1 & 0 & 0 & 0 & 1 & 0 & 0 & 0 & 0 & 0 & 0 & 0 & 0 & 0\\
          1 & 0 & 0 & 0 & 0 & 0 & 0 & 1 & 0 & 0 & 0 & 0 & 0 & 0\\
          0 & 0 & 0 & 0 & 0 & 0 & 1 & 0 & 1 & 0 & 0 & 0 & 0 & 0\\
          0 & 0 & 0 & 0 & 0 & 0 & 0 & 1 & 0 & 1 & 0 & 0 & 0 & 0\\
          0 & 1 & 0 & 0 & 0 & 0 & 0 & 0 & 1 & 0 & 0 & 0 & 0 & 0\\
          a_{111}^\sigma & 0 & 0 & 0 & 0 & 0 & 0 & 0 & 0 & 0 & 0 & 0 & a_{1113}^\sigma & 0\\
          0 & a_{212}^\sigma & 0 & 0 & 0 & 0 & 0 & 0 & 0 & 0 & 0 & 0 & 0 & a_{1214}^\sigma\\
          0 & 0 & 0 & 0 & 0 & 0 & 0 & 0 & 0 & 0 & a_{1113}^\sigma & 0 & 0 & a_{1314}^\sigma\\
          0 & 0 & 0 & 0 & 0 & 0 & 0 & 0 & 0 & 0 & 0 & a_{1214}^\sigma & a_{1314}^\sigma & 0\\
        \end{array}
      \right).$$
Applying elementary congruence matrix operations on $A((B(6,6,6),\sigma))$, it is congruent to
$$ \setlength{\arraycolsep}{10pt}\left(
        \begin{array}{cccccccccccccc}
          0 & -1 & 0 & 0 & 0 & 0 & 0 & 0 & 0 & 0 & 0 & 0 & 0 & 0\\
          -1 & 0 & 0 & 0 & 0 & 0 & 0 & 0 & 0 & 0 & 0 & 0 & 0 & 0\\
          0 & 0 & 0 & 1 & 0 & 0 & 0 & 0 & 0 & 0 & 0 & 0 & 0 & 0\\
          0 & 0 & 1 & 0 & 0 & 0 & 0 & 0 & 0 & 0 & 0 & 0 & 0 & 0\\
          0 & 0 & 0 & 0 & 0 & 1 & 0 & 0 & 0 & 0 & 0 & 0 & 0 & 0\\
          0 & 0 & 0 & 0 & 1 & 0 & 0 & 0 & 0 & 0 & 0 & 0 & 0 & 0\\
          0 & 0 & 0 & 0 & 0 & 0 & 0 & 1 & 0 & 0 & 0 & 0 & 0 & 0\\
          0 & 0 & 0 & 0 & 0 & 0 & 1 & 0 & 0 & 0 & 0 & 0 & 0 & 0\\
          0 & 0 & 0 & 0 & 0 & 0 & 0 & 0 & 0 & 2 & 0 & 0 & 0 & 0\\
          0 & 0 & 0 & 0 & 0 & 0 & 0 & 0 & 2 & 0 & 0 & 0 & 0 & 0\\
          0 & 0 & 0 & 0 & 0 & 0 & 0 & 0 & 0 & 0 & 0 & \frac{-a_{111}^\sigma a_{212}^\sigma}{2} & 0 & 0\\
          0 & 0 & 0 & 0 & 0 & 0 & 0 & 0 & 0 & 0 & \frac{-a_{111}^\sigma a_{212}^\sigma}{2} & 0 & 0 & 0\\
          0 & 0 & 0 & 0 & 0 & 0 & 0 & 0 & 0 & 0 & 0 & 0 & 0 & \lambda_3\\
          0 & 0 & 0 & 0 & 0 & 0 & 0 & 0 & 0 & 0 & 0 & 0 & \lambda_3 & 0\\
        \end{array}
      \right),$$
where $\lambda_3=\frac{2a_{1113}^\sigma a_{1214}^\sigma}{a_{111}^\sigma a_{212}^\sigma}+{a_{1314}^\sigma}$. Since $\frac{2a_{1113}^\sigma a_{1214}^\sigma}{a_{111}^\sigma a_{212}^\sigma}\not=1,-1$, $i_-((B(6,6,6),\sigma))=7>\lceil\frac{10}{2}\rceil+1=6$ following from Lemma \ref{lem-2-0}, a contradiction. We are done.
\end{proof}


\begin{thebibliography}{99}

\bibitem{F.Belardo} F. Belardo, S. Cioab$\acute{a}$, J. Koolen, J.F. Wang, Open problem in the spectral theory of signed graphs, Art Discrtete Applied Mathematics 1 (2018) 2-24.
 https://doi.org/ 10.48550/ arXiv.1907.04349
\bibitem{M.Brunetti} M. Brunetti, Z. Stani$\acute{c}$, Ordering signed graphs with large index, Ars Mathematica Contemporanea, 22(2022) P4.05. https://doi.org/10.26493/1855-3974.2714.9b3
\bibitem{D.Cvetkovi} D. Cvetkovi\'{c}, M. Doob, H. Sachs, Spectra of Graphs: Theory and Application. Academic Press, New York, 1980.
\bibitem{Fang.D} F. Duan, Y.H. Yang, Triangle-free signed graphs with small negative inertia index, Discrete Applied Mathematics, 357(2024) 135-142. https://doi.org/10.1016/j.dam. 2024.06.012
\bibitem{Fang.D1} F. Duan, Characterizing the negative inertia index of connected graphs in terms of their girth, Discrete Mathematics, 347(2024) 113997. https://doi.org/10.1016/j. disc.2024.113997
\bibitem{Fang.D2} F. Duan, Q. Yang, On graphs with girth $g$ and positive inertia index of $\lceil\frac{g}{2}\rceil$-1 and $\lceil\frac{g}{2}\rceil$, Linear Algebra and its Applications, 683(2024) 98-110. https://doi.org/ 10.1016/j.laa.2023.12.001
\bibitem{L.D} Liu B, Duan F. Connected signed graphs with given inertia indices and given girth[J]. arxiv preprint arxiv:2505.08539, 2025.
\bibitem{K.S} K S. Signed graphs $G^\sigma$ with nullity $n(G^\sigma)-g(G^\sigma)-1$. Linear Algebra and Its Applications, 2026, 728: 47-62. https://doi.org/10.1016/j.laa.2025.08.017.
\bibitem{FanY.Z} Y.Z. Fan, Y. Wang, Y. Wang, A note on the nullity of unicyclic signed graphs, Linear Algebra and its Applications, 438(2013) 1193-1200. https://doi.org/10.1016/j. laa.2012.08.027
\bibitem{W.H.Haemers} W.H. Haemers, H. Topcu, On signed graphs with at most two eigenvalues unequal to $\pm1$, Linear Algebra and its Applications, 670(2023) 68-77. https://doi.org /10.1016/j.laa.2023.04.001
\bibitem{F.Harary} F. Harary, On the notion of balanced in a signed graph, Michigan Math. J. 2 (1) (1953) 143-146. https://doi.org/10.1016/ j.laa.2023.04.001
\bibitem{R.A.Horn} R.A. Horn, C.R. Johnson, Matrix Analysis, Cambridge University Press, 1985.
\bibitem{Yaoping.Hou} Y.P. Hou, J.S. Li, Y.L. Pan, On the Laplacian eigenvalues of signed graphs, Linear Multilinear Algebra, 51(1)(2003) 21-30. https://doi.org/10.1080/03081080310000
    53611
\bibitem{M.R.Oboudi3} M.R. Oboudi, Characterization of graphs with exactly two non-negative eigenvalues, Ars Mathematica Contemporanea, 12(2017) 271-286. https://doi.org/10.26493/1855-3974.1077.5b6
\bibitem{Q.Wu} Q. Wu, Y. Lu, B.S. Tam, On connected signed graphs with rank equal to girth, Linear Algebra and its Applications 651 (2022) 90-115. https://doi.org/10.1016/ j.laa.2022.06.019
\bibitem{GuiHai.Yu} G. H. Yu, L. H. Feng, Q. W. Wang, Bicyclic graphs with small positive index of inertia, Linear Algebra and its Applications, 438 (2013) 2036-2045. https://doi.org/10.10 16/j.laa.2012.09.031
\bibitem{G.H.Yu} G.H. Yu, X.D. Zhang, L.H. Feng, The inertia of weighted unicyclic graphs, Linear Algebra and its Applications, 44(2014) 130-152. https://doi.org/10.1016/j.laa .2014.01.023
\bibitem{G.H.Yu1} G.H. Yu, L.H. Feng, Q.W. Wang, A. Ili$\acute{c}$, The minimal positive index of inertia of unicyclic signed graphs, Ars Combinatoria, 117(2014) 245-255.



\end{thebibliography}
\end{document}